\documentclass[12pt,a4paper]{amsart}
\usepackage{color,graphicx,fullpage,amssymb,tikz}
\usepackage{url}
\usepackage[colorlinks=true]{hyperref}
\newtheorem{theorem}{Theorem}[section]
\newtheorem{lemma}[theorem]{Lemma}
\newtheorem{proposition}[theorem]{Proposition}
\newtheorem{corollary}[theorem]{Corollary}

\newtheorem{maintheorem}{Theorem}

\theoremstyle{definition}
\newtheorem{definition}[theorem]{Definition}
\newtheorem{remark}[theorem]{Remark}

\newcommand{\Q}{\mathbb{Q}}
\newcommand{\Qp}{\mathbb{Q}_p}
\newcommand{\Cp}{\mathbb{C}_p}
\newcommand{\Zp}{\mathbb{Z}_p}
\newcommand{\N}{\mathbb{N}}
\newcommand{\R}{\mathbb{R}}
\newcommand{\Z}{\mathbb{Z}}
\newcommand{\C}{\mathbb{C}}
\newcommand{\dd}{\mathrm{d}}

\newcommand{\M}{\mathcal{M}}

\newcommand{\tr}{^{\mathrm{T}}}
\newcommand{\ocal}{\mathcal{O}}
\newcommand{\letnpos}{Let $n$ be a positive integer}
\newcommand{\letpprime}{Let $p$ be a prime number}
\newcommand{\ballcp}{\mathrm{B}_p^{\C}}
\renewcommand{\le}{\leqslant}
\renewcommand{\ge}{\geqslant}
\AtBeginEnvironment{pmatrix}{\let\frac\dfrac}
\DeclareMathOperator{\ord}{ord}
\DeclareMathOperator{\DSq}{DSq}
\DeclareMathOperator{\Ind}{Ind}
\numberwithin{equation}{section}

\newenvironment{enumerate-roman}{\begin{enumerate}
		
	}{\end{enumerate}}

\title{The contrast between real and $p$-adic integrability}

\author[Luis Crespo, \'Alvaro Pelayo]{Luis Crespo\,\,\,\,\,\, \'Alvaro Pelayo}

\address{Luis Crespo,
	Departamento de Matem\'{a}ticas, Estad\'{i}stica y Computaci\'{o}n, Universidad de Cantabria, Av.~de Los Castros 48, 39005 Santander, Spain}
\email{luis.cresporuiz@unican.es}

\address{\'Alvaro Pelayo,
	Facultad de Ciencias Matem\'aticas,
	Universidad Complutense de Madrid, 28040 Madrid, Spain, and Real Academia de Ciencias Exactas, F\'isicas y Naturales, Madrid, Spain}
\email{alvpel01@ucm.es}

\begin{document}
	
\maketitle

\vspace{-1em}

\begin{center}
	\emph{Dedicated with admiration to Juan Luis Vázquez, pioneer in analysis and differential equations, on his eightieth birthday.}
\end{center}

\begin{abstract}
	A fundamental question in $p$-adic symplectic geometry and $p$-adic Hamiltonian mechanics is to what extent we can learn new information about integrable systems with coefficients in the rational numbers $\Q$, if we view them as $p$-adic systems instead of real ones. We combine geometric analytic estimates with Galois theory in order to show that at a local level, these real and $p$-adic viewpoints give completely different information. Then we focus on understanding this local behavior for the $p$-adic case by proving effective methods to compute local normal forms. Our results are motivated by a program outlined by Pelayo-Voevodsky-Warren a decade ago, to develop a $p$-adic theory of integrable systems and symplectic geometry.
\end{abstract}

\section{Introduction}

One of the most intriguing questions at the intersection of symplectic geometry and number theory is to what extent the $p$-adic numbers capture phenomena outside of real symplectic geometry. Since symplectic geometry has its origins in classical mechanics, and since integrable systems \cite{PelVuN-integrable} are essentially physical systems with many conserved quantities, it is natural to ask if considering the same system with coefficients in $\R$ and $\Qp$ respectively, where $p$ is a prime number (provided of course it has rational coefficients, so it is ``both real and $p$-adic''), will lead us to see the same or different behaviors. Our main motivation is the following meta-goal:

\medskip
\noindent\textbf{Meta-goal (Real vs $p$-adic dynamics).} \textit{The qualitative and quantitative dynamics of an integrable Hamiltonian system with rational coefficients shifts completely, depending on whether we consider the system as a real system or a $p$-adic one, where $p$ is a prime number. Even more, the change also occurs if we change the prime number $p$ to a different one. Hence, the information we get about a real physical integrable system depends on the structure of the numerical system we use to measure.}

\medskip
We will establish the local version of this meta-goal (Theorems \ref{thm:indep2}, \ref{thm:indep3}), which essentially implies the global version too (nonetheless we believe that a global level even more drastic changes will occur) and present computationally effective methods to describe the quantitative aspects of the dynamics (Theorems \ref{thm:integrable1}, \ref{thm:integrable}). This implies that $\R$ and $\Qp$ capture fundamentally different information about the system, despite the fact that they both contain $\Q$ as a dense subset and that experiments can only measure rational quantities. Hence a $p$-adic version of symplectic geometry and Hamiltonian mechanics will be different from its real counterpart, at least from the point of view of integrable systems.

To be more concrete, our first main result, Theorem \ref{thm:indep2}, states that, for integrable systems in dimension $4$ (which implies the same for dimension $2$, see Theorem \ref{thm:indep1}) there is no relation between the real normal form and the $p$-adic normal form of an integrable system with rational coefficients (which makes sense in both contexts). That is, all combinations of a real and a $p$-adic normal form are possible, and even of a real normal form and several $p$-adic normal forms for different primes $p$. A weaker generalization to any dimension is provided as Theorem \ref{thm:indep3}. Then we introduce techniques to compute the local normal forms (see Appendix \ref{sec:classification} for a review of the known normal forms) in the extreme but fundamental case of corank $1$ critical points and some cases of greater corank, including all cases in dimension $4$. Theorem \ref{thm:integrable1} says that the local normal form of a $p$-adic integrable system with $n$ degrees of freedom (such as the Jaynes-Cummings model or the coupled angular momentum for $n=2$) at a corank $1$ critical point has a component of the form $x_1^2+c\xi_1^2$ in some local symplectic coordinates $(x_1,\xi_1,\ldots,x_n,\xi_n)$, with $c\in\Qp$ being a value for which a certain $p$-adic biquadratic polynomial has a zero. Theorem \ref{thm:integrable} gives a similar statement for some critical points with arbitrary corank.

\bigskip
\noindent \textit{Structure of the paper.} Section \ref{sec:main} states the main results Theorems \ref{thm:indep2}, \ref{thm:indep3}, \ref{thm:integrable1} and \ref{thm:integrable}. In Section \ref{sec:integrable} we prove results about approximation of $p$-adic eigenvectors and apply them to derive Theorems \ref{thm:integrable1} and \ref{thm:integrable}. In Section \ref{sec:lowdim} we prove Theorems \ref{thm:indep2} and \ref{thm:indep3}. The paper concludes with three brief appendices.

\section{Main results}\label{sec:main}

We state here our main results, Theorems \ref{thm:indep2}, \ref{thm:indep3}, \ref{thm:integrable1} and \ref{thm:integrable}. For a concise introduction to the basics of $p$-adic symplectic geometry and $p$-adic integrable systems see \cite[Section 3]{CrePel-JC}. In dimension $4$, each of the finite number of integrable systems $(\R^4,\dd x\wedge\dd\xi+\dd y\wedge\dd\eta)\to\R^2$ (respectively $((\Qp)^4,\dd x\wedge\dd\xi+\dd y\wedge\dd\eta)\to(\Qp)^2$) with exactly one critical point at the origin to which any other integrable system $(\R^4,\dd x\wedge\dd\xi+\dd y\wedge\dd\eta)\to\R^2$ (respectively $((\Qp)^4,\dd x\wedge\dd\xi+\dd y\wedge\dd\eta)\to(\Qp)^2$) can be converted to, by means of a local change of coordinates centered at the origin, is called a \emph{non-degenerate integrable real normal form} (respectively, \emph{a non-degenerate integrable $p$-adic normal form}). In the $p$-adic case the complete list of normal forms is explicitly given at \cite[Theorems A and B]{CrePel-williamson} and includes hundreds of such systems. In the real case see \cite{Williamson} or \cite[Theorem 4.2.1]{PelVuN-integrable}. In both cases the list of normal forms is finite in each dimension, but in the $p$-adic case it grows exponentially with the dimension \cite[Theorem C]{CrePel-williamson2}.

\begin{maintheorem}[Integrable systems on $p$-adic and real symplectic manifolds display completely independent behaviors, case of dimension $4$]\label{thm:indep2}
	Let $r\in\{0,1\}$. Let $k$ be a positive integer. Let $p_1,\ldots,p_k$ be different prime numbers. Let $F_0:(\R^4,\dd x\wedge\dd\xi+\dd y\wedge\dd\eta)\to\R^2$ be a non-degenerate integrable real normal form of rank $r$ and let $F_i:((\Q_{p_i})^4,\dd x\wedge\dd\xi+\dd y\wedge\dd\eta)\to(\Q_{p_i})^2$ a non-degenerate integrable $p_i$-adic normal form of rank $r$ for each $i\in\{1,\ldots,k\}$. Then there exist polynomials $f_1,f_2\in\Q[x,\xi,y,\eta]$ which satisfy all of the following conditions simultaneously:
	\begin{enumerate}
		\item $(f_1,f_2):\R^4\to\R^2$ is a real integrable system;
		\item $(f_1,f_2):(\Q_{p_i})^4\to(\Q_{p_i})^2$ is a $p_i$-adic integrable system for every $i\in\{1,\ldots,k\}$;
		\item as a real integrable system $(f_1,f_2):\R^4\to\R^2$ has a non-degenerate critical point at $(0,0,0,0)\in\R^4$ and the integrable real normal form of $(f_1,f_2)$ at $(0,0,0,0)$ is $F_0$;
		\item for every $i\in\{1,\ldots,k\}$, as a $p_i$-adic integrable system $(f_1,f_2):(\Q_{p_i})^4\to(\Q_{p_i})^2$ has a non-degenerate critical point at $(0,0,0,0)\in(\Q_{p_i})^4$ and the integrable $p_i$-adic normal form of $(f_1,f_2)$ at $(0,0,0,0)$ is $F_i$ for each $i\in\{1,\ldots,k\}$.
	\end{enumerate}
\end{maintheorem}

This result can be generalized to higher dimensions. In order to do that, we suppose that each pair of components of all the normal forms form a $4$-dimensional integrable system to which we can apply separately Theorem \ref{thm:indep2}. Obtaining a complete analog of Theorem \ref{thm:indep2} in every dimension would first require developing the Weierstrass-Williamson classification in arbitrary dimension, which is out of reach at the moment.

\begin{maintheorem}[Integrable systems on $p$-adic and real symplectic manifolds display completely independent behaviors, case of dimension $2n$]\label{thm:indep3}
	Let $s$ be a positive integer. Let $r\in\{0,\ldots,2s-1\}$. Let $k$ be a positive integer. Let $p_1,\ldots,p_k$ be different prime numbers. Let $F_0:(\R^{4s},\omega_0)\to\R^{2s}$ be a non-degenerate integrable real normal form of rank $r$ and let $F_i:((\Q_{p_i})^{4s},\omega_0)\to(\Q_{p_i})^{2s}$ a non-degenerate integrable $p_i$-adic normal form of rank $r$ for each $i\in\{1,\ldots,k\}$. Suppose that, for any $j\in\{1,\ldots,s\}$, the components in positions $2j-1$ and $2j$ only depend on the variables $x_{2j-1}$, $\xi_{2j-1}$, $x_{2j}$ and $\xi_{2j}$. Then there exist polynomials $f_1,\ldots,f_{2s}\in\Q[x_1,\xi_1,\ldots,x_{2s},\xi_{2s}]$ which satisfy all of the following conditions simultaneously:
	\begin{enumerate}
		\item $(f_1,\ldots,f_{2s}):\R^{4s}\to\R^{2s}$ is a real integrable system;
		\item $(f_1,\ldots,f_{2s}):(\Q_{p_i})^{4s}\to(\Q_{p_i})^{2s}$ is a $p_i$-adic integrable system for every $i\in\{1,\ldots,k\}$;
		\item as a real integrable system $(f_1,\ldots,f_{2s}):\R^{4s}\to\R^{2s}$ has a non-degenerate critical point at $0\in\R^{4s}$ and the integrable real normal form of $(f_1,\ldots,f_{2s})$ at $0$ is $F_0$;
		\item for every $i\in\{1,\ldots,k\}$, as a $p_i$-adic integrable system $(f_1,\ldots,f_{2s}):(\Q_{p_i})^{4s}\to(\Q_{p_i})^{2s}$ has a non-degenerate critical point at $0\in(\Q_{p_i})^{4s}$ and the integrable $p_i$-adic normal form of $(f_1,\ldots,f_{2s})$ at $0$ is $F_i$ for each $i\in\{1,\ldots,k\}$.
	\end{enumerate}
\end{maintheorem}

Hence $\R$ and $\Qp$ capture fundamentally different information of an integrable system, even though they both contain $\Q$ as a dense subset and that experiments can only measure rational quantities. This suggests that $p$-adic symplectic geometry and $p$-adic Hamiltonian mechanics offer complementary viewpoints on the same physical phenomena, and the $p$-adic view shows new aspects.

Next we focus on computing effectively $p$-adic normal forms.

\begin{definition}[Aligned symplectic coordinates]
	\letnpos. \letpprime. Let $(M,\omega)$ be a $p$-adic analytic symplectic manifold of dimension $2n$. Let $F=(f_1,\ldots,f_n):(M,\omega)\to(\Qp)^n$ be a $p$-adic analytic integrable system. Let $m$ be a critical point of $F$ and let $r$ be the rank of $m$. We say that $F$ is in \emph{separated form at $m$} if $\dd f_1(m)=\ldots=\dd f_{n-r}(m)=0$. Suppose now that $F$ is in separated form at $m$. Let $U$ be an open set which contains $m$. We say that $(x_1,\xi_1,\ldots,x_n,\xi_n)$ are \emph{aligned local symplectic coordinates} for $F$ on $U$ centered at $m$ if $\dd\xi_i(m)=\dd f_i(m)$ for all $i>n-r$ and $\omega=\sum_{i=1}^n\dd x_i\wedge\dd\xi_i$ on $U$.
\end{definition}

Aligned local coordinates always exist, as we prove in Appendix \ref{sec:proof-aligned}.

\begin{maintheorem}[Normal form computation for corank $1$ critical points]\label{thm:integrable1}
	\letnpos. \letpprime. Let $(M,\omega)$ be a $p$-adic analytic symplectic manifold of dimension $2n$. Let $F=(f_1,\ldots,f_n):(M,\omega)\to(\Qp)^n$ be a $p$-adic analytic integrable system. Let $m$ be a corank $1$ non-degenerate critical point of $F$ and let $F'=(f_1',\ldots,f_n')$ be a separated form of $F$ at $m$. Suppose that $m$ is a non-degenerate critical point of $f_1'$. Then there exists an open neighborhood $U$ of $m$ and aligned symplectic coordinates $(x_1,\xi_1,\ldots,x_n,\xi_n)$ for $F'$ on $U$ centered at $m$ such that
	$F|_U=F'_U=(x_1^2+c\xi_1^2,\xi_2,\ldots,\xi_n),$
	where $c\in\Qp$ is any element of $\Qp$ such that
	\[\left(\frac{\partial^2 f_1'}{\partial x_1'^2}(m)\right)^2cr^4+2\left(\frac{\partial^2 f_1'}{\partial x_1'^2}(m)\right)^2c^2r^2s^2+\left(\frac{\partial^2 f_1'}{\partial x_1'^2}(m)\right)^2c^3s^4\]\[+\left(\frac{\partial^2 f_1'}{\partial x_1'\partial \xi_1'}(m)\right)^2-\frac{\partial^2 f_1'}{\partial x_1'^2}(m)\frac{\partial^2 f_1'}{\partial \xi_1'^2}(m)\in\Qp[r,s]\]
	has a zero in $(\Qp)^2$.
\end{maintheorem}

If the corank of the critical point is greater than $1$, unlike what happens in the case of real integrable systems, in order to compute its normal form, we need the know the eigenvectors of a certain matrix, and not only the eigenvalues. This makes the $p$-adic context subtle. Given an $m$-by-$n$ matrix $A$ and two subsets $I\subset\{1,\ldots,m\}$ and $J\subset\{1,\ldots,n\}$, we denote by $A_{IJ}$ the submatrix of $A$ obtained taking the rows indexed by $I$ and the columns indexed by $J$. Likewise, $A_{\cdot J}$ represents the submatrix obtained taking the columns indexed by $J$ and all the rows, and $A_{I\cdot}$ the submatrix obtained taking the rows indexed by $I$ and all the columns. In particular, for $1\le i\le m$, $A_{i\cdot}$ is the $i$-th row of $A$, and for $1\le j\le n$, $A_{\cdot j}$ is the $j$-th column of $A$.

\begin{definition}[Independence number of a $p$-adic matrix]\label{def:indep}
	Let $m,n$ be positive integers such that $m\le n$. \letpprime. Let $\Cp$ be the field of complex $p$-adic numbers. Let $|\cdot|_p$ be the $p$-adic absolute value and let $\|\cdot\|_p$ be the $p$-adic norm. Let $A\in\M_{m\times n}(\Cp)$ be an $m$-by-$n$ matrix with coefficients in $\Cp$. Suppose that no row of $A$ is identically zero. The \emph{independence number $\Ind_p(A)$ of $A$} is the largest of the $p$-adic absolute values of the determinants of the $m$-by-$m$ minors of $A$ divided by the product of the $p$-adic norms of the rows of $A$. If a row of $A$ is identically zero, the independence number is defined to be zero. That is, $\Ind_p(A)$ is given by the formula:
	\[\Ind_p(A)=\begin{cases}
		\displaystyle\frac{\max_{|J|=m}|\det A_{\cdot J}|_p}{\prod_{i=1}^m\|A_{i\cdot}\|_p} & \text{if no row of $A$ is identically }0; \\
		0 & \text{otherwise.}
	\end{cases}\]
\end{definition}

\begin{definition}[Almost eigenvector]\label{def:almost}
	Let $\ell$ be a positive integer. \letpprime. Let $\mathrm{I}_\ell$ be the identity matrix of size $\ell$ and let $\mathrm{e}_i$ be the $i$-th vector of the canonical basis of $(\Qp)^\ell$. Let $q=8$ if $p=2$ and otherwise $q=p$. Let $A\in\M_\ell(\Qp)$. Let $\lambda$ be an eigenvalue of $A$ such that $\lambda\notin\Qp$ and $\lambda^2\in\Qp$. Let $I$ be a subset of $\{1,\ldots,\ell\}$ of size $\ell-1$. Suppose that $\Ind_p(\lambda \mathrm{I}_\ell-A)_{I\cdot}\ne 0$. Let $h:(\Cp)^\ell\to\R$ be given by
	\[h(v)=h(v_1,\ldots,v_\ell)=\frac{1}{\Ind_p(\lambda \mathrm{I}_\ell-A)_{I\cdot}\|v\|_p}\max_{i \in I}\frac{|\lambda v_i-A_{i\cdot}v|_p}{\|\lambda \mathrm{e}_i-A_{i\cdot}\|_p}.\]
	Let $\Omega\in\M_\ell(\Qp)$. Let $k$ be the largest absolute value of an entry of $\Omega$. We say that $v_0\in(\Cp)^\ell$ is an $(A,\lambda,\Omega,I)$-\textit{almost eigenvector} if
	$h(v_0)\le \min\left\{1,\frac{|v_0\tr\Omega \bar{v}_0|_p}{kq\|v_0\|_p^2}\right\},$
	where $x\mapsto\bar{x}$ is any automorphism of $\Cp$ that fixes $\Qp$ and sends $\lambda$ to $-\lambda$. We say that $v_0$ is an $(A,\lambda,\Omega)$-\textit{almost eigenvector} if there exists a subset $I$ of $\{1,\ldots,\ell\}$ of size $\ell-1$ such that $v_0$ is an $(A,\lambda,\Omega,I)$-almost eigenvector.
\end{definition}

The following statement does not involve aligned symplectic coordinates, however, the proof will require them.

\begin{maintheorem}[Normal form computation for critical points based on almost eigenvectors]\label{thm:integrable}
	\letnpos. \letpprime. Let $(M,\omega)$ be a $p$-adic analytic symplectic manifold of dimension $2n$. Let $F=(f_1,\ldots,f_n):(M,\omega)\to(\Qp)^n$ be a $p$-adic analytic integrable system. Let $m$ be a non-degenerate critical point of $F$. Let $a_1,\ldots,a_n\in\Qp$ such that $m$ is a non-degenerate critical point of $\sum_{i=1}^n a_if_i$. Let $\Omega$ be the matrix of $\omega_m$ and $A=\Omega^{-1}\dd^2 \sum_{i=1}^n a_if_i$. Suppose that there exists an eigenvalue $\lambda$ of $A$ such that $\lambda\notin\Qp$ and $\lambda^2\in\Qp$, and $v_0\in(\Cp)^{2n}$ which is an $(A,\lambda,\Omega)$-almost eigenvector according to Definition \ref{def:almost} with $\ell=2n$. Let $x\mapsto\bar{x}$ be an automorphism of $\Cp$ that fixes $\Qp$ and sends $\lambda$ to $-\lambda$. Let $k=v_0\tr\Omega \bar{v}_0$. Then there exist an open neighborhood $U$ of $m$, and local symplectic coordinates $(x_1,\xi_1,\ldots,x_n,\xi_n)$ on $U$ centered at $m$ such that, in these coordinates, one of the components of $F$ has the form $(x_1,\xi_1,\ldots,x_n,\xi_n)\mapsto x_1^2+c\xi_1^2$, where $c$ is any element of $\Qp$ such that
	\[p_{c,\lambda,k}(r,s)=(ck^2)r^4+(2c^2k^2)r^2s^2+(c^3k^2)s^4+4\lambda^4\]
	has a zero in $(\Qp)^2$.
\end{maintheorem}

For an introduction to the $p$-adic numbers and their role in geometry, we recommend \cite{Gouvea,Schneider}. $p$-adic geometry has been extensively used in physics, see for example \cite{BreFre,DKKV,DKKVZ}. See also \cite{Gosson,Dragovich-quantum,RTVW,VlaVol} for an application to quantum mechanics, \cite{BFOW,FreOls,FreWit,FGZ,GarLop,Volovich} for applications to string theory and \cite{CLH} for an application to the theory of black holes.

We recommend the books by Marsden-Ratiu \cite{MarRat} and Ortega-Ratiu \cite{OrtRat} for an introduction to the ideas of classical mechanics and symplectic geometry, with an angle which is very appropriate for the study of integrable systems; see also the article by Ratiu-Wacheux-Zung \cite{RWZ} for a recent treatment of how the critical points of a real integrable system affect the singular affine structure they determine (we believe that such ideas should also play a role in the $p$-adic situation, when trying to understand affine structures). Integrable systems are not only important in physics, but also in robotics \cite{RatZun}, semiclassical analysis and the study of Lagrangian submanifolds \cite{Reshetikhin,Weinstein}, and multiple other areas within mathematics and physics.

The $p$-adic numbers also appear in other areas of physics/mathematics of a high current interest. For example, there are connections between tensor networks \cite{CPSV} and $p$-adic fields, see the paper \cite{HMSS}. We hope that ideas from $p$-adic symplectic geometry could be useful in some of these areas. $p$-adic integrable systems belong to the emerging field of $p$-adic symplectic geometry, which has been studied by the authors and their collaborators in previous papers \cite{CrePel-JC,CrePel-nonsqueezing,CrePel-williamson,CrePel-williamson2,CrePel-angular,CrePel-Darboux,PVW}, see also \cite{HuHu,Zelenov} for some symplectic constructions over the $p$-adic field.

\section{Proofs of Theorems \ref{thm:integrable1} and \ref{thm:integrable}}\label{sec:integrable}

In this section we prove Theorem \ref{thm:integrable1}, which gives us a method to compute the normal form of corank $1$ critical points, and Theorem \ref{thm:integrable}, which shows that, for some critical points, an approximation of the eigenvectors is sufficient in order to compute the normal form.

\subsection{Review of the Weierstrass-Williamson classification}

\begin{definition}[Non-degenerate critical point of $p$-adic analytic function, symplectic sense]\label{def:nondeg}
	Let $n$ be a positive integer. \letpprime. Let $(M,\omega)$ be a $2n$-dimensional $p$-adic analytic symplectic manifold. Let $f:M\to\Qp$ be a $p$-adic analytic function and let $m\in M$ be a critical point of $f$ (i.e. $\dd f(m)=0$). Let $\Omega$ be the matrix of $\omega$. We say that $m$ is \emph{non-degenerate} if the eigenvalues of $\Omega^{-1}\dd^2 f(m)$ are all distinct.
\end{definition}

See Proposition \ref{prop:nondeg} for the relation between this and the usual definition of non-degenerate critical point.

\begin{definition}[$p$-adic non-degenerate critical point {\cite[Definition 5.2]{CrePel-williamson}}]\label{def:nondeg-integrable}
	\letnpos. \letpprime. Let $(M,\omega)$ be a $p$-adic analytic symplectic manifold of dimension $2n$, $F=(f_1,\ldots,f_n):M\to(\Qp)^n$ a $p$-adic analytic integrable system and $m$ a rank $0$ critical point of $F$. Let $\Omega$ be the matrix of $\omega_m$. We say that $m$ is \emph{non-degenerate} if for every generic choice of $a_i\in\Qp$ the $p$-adic matrix
	$\Omega^{-1}\sum_{i=1}^{n}a_i\dd^2f_i$
	has $2n$ different eigenvalues. If $m$ has rank $r>0$ and $\dd f_1,\ldots,\dd f_r$ are linearly independent, we consider local symplectic coordinates $(x_1,\xi_1,\ldots,x_n,\xi_n)$ centered at $m$ such that $\dd f_i=\dd \xi_i$ for all $i\in\{1,\ldots,r\}$. Then we say that $m$ is \emph{non-degenerate} if the $p$-adic matrix above has $2(n-r)$ different eigenvalues when $\Omega$ and the Hessian matrices are restricted to the coordinates $x_{r+1},\xi_{r+1},\ldots,x_n,\xi_n.$
\end{definition}

For each prime number $p$ such that $p\equiv 1\mod 4$, let $c_0$ be the smallest quadratic non-residue modulo $p$. Following \cite[Definition 1.1]{CrePel-williamson}, we consider the sets
\begin{equation}\label{eq:X}
	X_p=\begin{cases}
		\{1,c_0,p,c_0p,(c_0)^2p,(c_0)^3p,c_0p^2\} & \text{if }p\equiv 1\mod 4; \\
		\{1,-1,p,-p,p^2\} & \text{if }p\equiv 3\mod 4; \\
		\{1,-1,2,-2,3,-3,6,-6,12,-18,24\} & \text{if }p=2,
	\end{cases}
\end{equation}
\[Y_p=\begin{cases}
	\{c_0,p,c_0p\} & \text{if }p\equiv 1\mod 4; \\
	\{-1,p,-p\} & \text{if }p\equiv 3\mod 4; \\
	\{-1,2,-2,3,-3,6,-6\} & \text{if }p=2,
\end{cases}\]
and the functions $\mathcal{C}_i^k:Y_p\times(\Qp)^4\to\Qp$ and $\mathcal{D}_i^k:Y_p\times(\Qp)^4\to\Qp$, for $k\in\{1,2\}$, $i\in\{0,1,2\}$, by
\begin{align*}
	&\mathcal{C}_0^1(c,t_1,t_2,a,b)=\frac{ac}{2(c-b^2)},
	&&\mathcal{C}_1^1(c,t_1,t_2,a,b)=\frac{b}{b^2-c}, \\
	&\mathcal{C}_2^1(c,t_1,t_2,a,b)=\frac{1}{2a(c-b^2)},
	&&\mathcal{C}_0^2(c,t_1,t_2,a,b)=\frac{abc}{2(b^2-c)},\\
	&\mathcal{C}_1^2(c,t_1,t_2,a,b)=\frac{c}{c-b^2},
	&&\mathcal{C}_2^2(c,t_1,t_2,a,b)=\frac{b}{2a(b^2-c)},\\
	&\mathcal{D}_0^1(c,t_1,t_2,a,b)=-\frac{t_1+bt_2}{2a},
	&&\mathcal{D}_1^1(c,t_1,t_2,a,b)=-bt_1-ct_2,\\
	&\mathcal{D}_2^1(c,t_1,t_2,a,b)=-\frac{ac(t_1+bt_2)}{2},
	&&\mathcal{D}_0^2(c,t_1,t_2,a,b)=-\frac{bt_1+ct_2}{2a},\\
	&\mathcal{D}_1^2(c,t_1,t_2,a,b)=-c(t_1+bt_2),
	&&\mathcal{D}_2^2(c,t_1,t_2,a,b)=-\frac{ac(bt_1+ct_2)}{2}.
\end{align*}

\begin{theorem}[{\cite[first part of Theorem A]{CrePel-williamson}}]\label{thm:williamson}
	The local normal form of a $p$-adic analytic integrable system on a $4$-dimensional $p$-adic analytic symplectic manifold at a rank $0$ non-degenerate critical point is of the form $(g_1,g_2)$, where $g_1$ and $g_2$ can take one of the following three possible forms:
	\begin{enumerate}
		\item $g_1(x,\xi,y,\eta)=x^2+c_1\xi^2,g_2(x,\xi,y,\eta)=y^2+c_2\eta^2$, for some $c_1,c_2\in X_p$;
		\item $g_1(x,\xi,y,\eta)=x\eta+cy\xi,g_2(x,\xi,y,\eta)=x\xi+y\eta$, for some $c\in Y_p$;
		\item
		\begin{equation}\label{eq:class3}
			g_k(x,\xi,y,\eta)=\sum_{i=0}^{2}\mathcal{C}_i^k(c,t_1,t_2,a,b)x^iy^{2-i}+\sum_{i=0}^{2}\mathcal{D}_i^k(c,t_1,t_2,a,b)\xi^i\eta^{2-i},
		\end{equation}
		for $k\in\{1,2\}$, where $c$, $t_1$, $t_2$, $a$ and $b$ can take a fixed set of possible values (see \cite[Table 1]{CrePel-williamson} for the full specification).
	\end{enumerate}
	Here $(x,\xi,y,\eta)$ are $p$-adic local symplectic coordinates centered at the critical point of the $p$-adic analytic integrable system, that is, the symplectic form $\omega_0$ has the form $\dd x\wedge\dd\xi+\dd y\wedge\dd\eta$ in this coordinates.
\end{theorem}

\begin{figure}
	\begin{tikzpicture}
		\node at (0,0) {\includegraphics[trim=10cm 2cm 10cm 2cm,clip,scale=0.5]{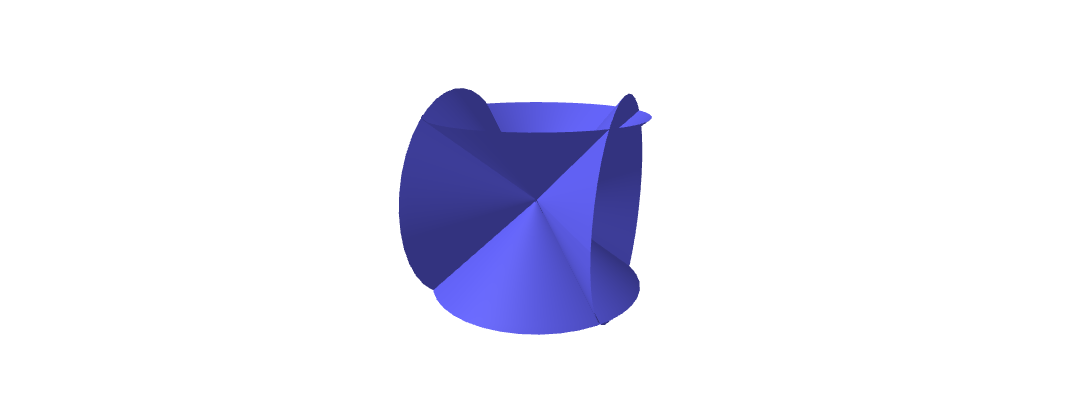}};
		\node at (0,2.2) {$x=\xi,y=\eta$};
		\node at (0,-2.7) {$x=-\xi,y=-\eta$};
		\node[right] at (2,0) {$x=\xi,y=-\eta$};
		\node[left] at (-2.5,0) {$x=-\xi,y=\eta$};
	\end{tikzpicture}
	\caption{Symbolic representation of $2$-dimensional fiber of the case $c=-1$ of point (i) of Theorem \ref{thm:integrable}, which is symplectically equivalent to $c=1$ if $p\equiv 1\mod 4$. The four ``cones'' are $2$-dimensional planes in $4$-dimensional space. See also Proposition \ref{prop:integrable}.}
	\label{fig:fibers1}
\end{figure}

\subsection{Proof of Theorem \ref{thm:integrable1}}

We prove here this result which implies Theorem \ref{thm:integrable1}:

\begin{theorem}\label{thm:integrable3}
	\letnpos. \letpprime. Let $(M,\omega)$ be a $p$-adic analytic symplectic manifold of dimension $2n$. Let $F=(f_1,\ldots,f_n):(M,\omega)\to(\Qp)^n$ be a $p$-adic analytic integrable system. Let $m$ be a corank $1$ non-degenerate critical point of $F$ and let $F'=(f_1',\ldots,f_n')$ be a separated form of $F$ at $m$ given by Lemma \ref{lemma:aligned}. Suppose that $m$ is a non-degenerate critical point of $f_1'$ in the symplectic sense. The following hold.
	\begin{enumerate-roman}
		\item Let $X_p$ be as in equation \eqref{eq:X}. Then there exists a unique $c\in X_p$, an open set $U$ containing $m$, and aligned local symplectic coordinates $(x_1,\xi_1,\ldots,x_n,\xi_n)$ for $F'$ on $U$ centered at $m$ such that the normal form of $F$ (which coincides with the normal form of $F'$) at $m$ is $(x_1^2+c\xi_1^2,\xi_2,\ldots,\xi_n)$.
		\item Let $(x_1',\xi_1',\ldots,x_n',\xi_n')$ be any aligned local symplectic coordinates for $F'$ on $U$ centered at $m$ (not necessarily those from part \textup{(i)}). Let
		\[\alpha=\frac{\partial^2 f_1'}{\partial x_1'^2}(m),\qquad \beta=\frac{\partial^2 f_1'}{\partial x_1'\partial \xi_1'}(m),\qquad \gamma=\frac{\partial^2 f_1'}{\partial \xi_1'^2}(m).\]
		The parameter $c\in X_p$ in part \textup{(i)} is the element $c\in X_p$ for which the $p$-adic biquadratic polynomial $p_{\alpha,\beta,\gamma,c}\in\Qp[r,s]$:
		\[p_{\alpha,\beta,\gamma,c}(r,s)=(\alpha^2c)r^4+(2\alpha^2c^2)r^2s^2+(\alpha^2c^3)s^4+(\beta^2-\alpha\gamma)\]
		has a zero in $(\Qp)^2$.
	\end{enumerate-roman}
\end{theorem}

\begin{proof}
	By Lemma \ref{lemma:aligned}, there exist aligned local symplectic coordinates $(x_1',\xi_1',\ldots,x_n',\xi_n')$ for the separated form $F'$ of $F$ at $m$. We take $x_i=x_i'$ and $\xi_i=\xi_i'$ for $i\ge 2$, so we are now essentially working in dimension $2$ and \cite[Theorem D]{CrePel-williamson} implies part (i). In particular, it implies that $c$ exists and it is unique.

We now prove part (ii). Let $H_1$ be the first two rows and columns of the Hessian of $f_1'$ in the aligned local symplectic coordinates $(x_1',\xi_1',\ldots,x_n',\xi_n')$:
\[H_1=\begin{pmatrix}
	\alpha & \beta \\
	\beta & \gamma
\end{pmatrix},\]
and let $\lambda$ and $-\lambda$ be the eigenvalues of $\Omega_0^{-1}H_1$. We have that $\lambda^2=\beta^2-\alpha\gamma$ and the eigenvector is
\[v=\begin{pmatrix}
	\lambda-\beta \\
	\alpha
\end{pmatrix}.\]
By \cite[Proposition 3.6]{CrePel-williamson2}, the singular component of the normal form (the rest are $\xi_i$ for each $i>1$) is $x_1^2+c\xi_1^2$ if and only if $\lambda^2=-ct^2$ for some $t\in \Qp$ and
\[\frac{2t\lambda}{v\tr\Omega_0 \bar{v}}\in\DSq(\Qp,-\lambda^2).\]
We have that
\begin{align*}
	v\tr\Omega_0\bar{v} & =\alpha(\lambda-\beta)-\alpha(-\lambda-\beta) \\
	& =2\alpha\lambda
\end{align*}
and
\begin{align*}
	\DSq(\Qp,-\lambda^2) & =\DSq(\Qp,ct^2) \\
	& =\DSq(\Qp,c),
\end{align*}
so the condition reduces to
\[\frac{t}{\alpha}\in\DSq(\Qp,c)\Longleftrightarrow \frac{t}{\alpha}=r^2+cs^2\]
for some $r,s\in\Qp$. This implies
\begin{align*}
	\alpha\gamma-\beta^2 & =-\lambda^2 \\
	& =ct^2 \\
	& =c\alpha^2(r^2+cs^2),
\end{align*}
as we wanted.
\end{proof}

\subsection{Approximation of $p$-adic equations and eigenvectors}\label{sec:approx}

The following result about approximation of $p$-adic eigenvectors is the key part of the proof of Theorem \ref{thm:integrable}. It gives an upper bound on the distance from a given vector $v_0$ to en eigenvector of a given matrix $A$ in terms of the distance from $Av_0$ to $\lambda v_0$ and the independence number of $\lambda \mathrm{I}_n-A$, where $\lambda$ is an eigenvalue of $A$. This result is also of independent interest and can be applied to problems outside of $p$-adic integrable systems and $p$-adic symplectic geometry.

\begin{theorem}[Approximation of $p$-adic eigenvectors]\label{thm:eigen}
	\letnpos. \letpprime. Let $A\in\M_n(\Cp)$. Let $v_0=(v_{01},\ldots,v_{0n})\in(\Cp)^n$. Let $\lambda$ be an eigenvalue of $A$. Let $r$ be the rank of $\lambda \mathrm{I}_n-A$. Let $I\subset\{1,\ldots,n\}$ of size $r$. Suppose that $\Ind_p((\lambda\mathrm{I}_n-A)_{I\cdot})\ne 0$. Then there exist infinitely many vectors
	\[v\in\ballcp\left(v_0,\frac{1}{\Ind_p((\lambda\mathrm{I}_n-A)_{I\cdot})}\max_{i \in I}\frac{|\lambda v_{0i}-A_{i\cdot}v_0|_p}{\|\lambda \mathrm{e}_i-A_{i\cdot}\|_p}\right)\]
	such that $v$ is an eigenvector corresponding to $\lambda$.
\end{theorem}

See Figure \ref{fig:ball} for an illustration of Theorem \ref{thm:eigen}.

In order to prove Theorem \ref{thm:eigen}, we need a previous result. Recall that the independence number of a matrix used in the following result was defined in Definition \ref{def:indep}. Also, the notation $A_{I\cdot}$ was introduced before Definition \ref{def:indep}.

\begin{lemma}[Approximation of solutions to $p$-adic equations]\label{lemma:main}
	Let $m,n,r$ be positive integers with $1\le r\le m$. \letpprime. Let $A\in\M_{m\times n}(\Cp)$ with rank $r$. Let $v_0\in(\Cp)^n$. Let $I\subset\{1,\ldots,m\}$ be a set of size $r$. Suppose that $\Ind_p(A_{I\cdot})\ne 0$. Then
	\[\Big\{x\in(\Cp)^n:Ax=0\Big\}\cap\ballcp\left(v_0,\frac{1}{\Ind_p(A_{I\cdot})}\max_{i \in I}\frac{|A_{i\cdot}v_0|_p}{\|A_{i\cdot}\|_p}\right)\ne\varnothing.\]
	If $r=n$, $Ax=0$ has a unique solution and this intersection is a single point. Otherwise, it has dimension $n-r>0$ and it contains infinitely many points.
\end{lemma}

\begin{figure}
	\includegraphics{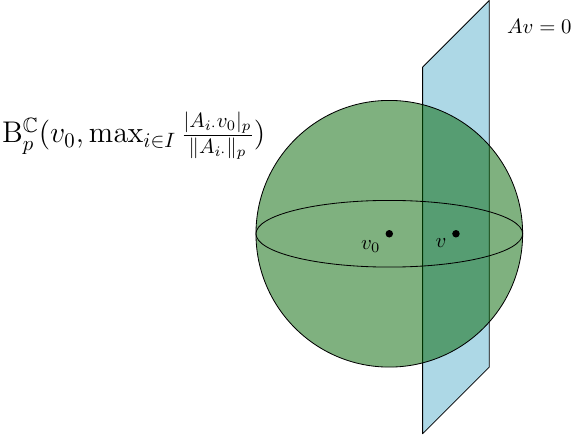}
	\caption{The equation $Av=0$ always has a solution in the ball $\ballcp(v_0,\frac{1}{c}\max_{i \in I}\frac{|A_{i\cdot}v_0|_p}{\|A_{i\cdot}\|_p})$, according to Lemma \ref{lemma:main}.}
	\label{fig:ball}
\end{figure}

\begin{proof}
	\textit{Step 1: Preliminary observation.} First we observe that multiplying a row of a matrix by a nonzero constant $k$ does not change its independence number, as given in Definition \ref{def:indep}. This is the case because the determinants of its minors are multiplied by $k$, so their $p$-adic absolute value is multiplied by $|k|_p$, and the product of the $p$-adic norms of its rows is multiplied by the same factor.
	
	\medskip
	\textit{Step 2: Construction of auxiliary matrix.} Let \[c=\Ind_p(A_{I\cdot}).\] Let $B$ be a matrix obtained from $A$ by scaling each row by some factor so that the $p$-adic norm of all nonzero rows of $B$ is $1$. In order to achieve norm $1$ in all the rows, we must make \[B_{i\cdot}=c_iA_{i\cdot}\] for some $c_i\in\Cp$ such that \[|c_i|_p=\frac{1}{\|A_{i\cdot}\|_p}.\]
	For example, we can write $\|A_{i\cdot}\|_p=p^k$ for $k\in\Q$. We can compute $p^k$ as a power of $p$ in $\overline{\Qp}$, and satisfies that $|p^k|_p=p^{-k}$, so we can take $B_{i\cdot}=p^kA_{i\cdot}.$ $B$ still has rank $r$, and since multiplying the rows of a matrix by arbitrary constants does not change its independence number, $\Ind_p(B_{I\cdot})=\Ind_p(A_{I\cdot})=c$. Since the rows have $p$-adic norm $1$, this in turn means that the greatest $p$-adic absolute value of the $r$-by-$r$ minors contained in those rows is $c$.
	
	\medskip
	\textit{Step 3: Construction of vector $v$.} Suppose without loss of generality that $I=\{1,\ldots,r\}$ and that the $p$-adic absolute value $c$ is attained at the minor formed by the first $r$ rows and columns of $B$. Then there exist matrices
	\[B_{11}\in\M_r(\Cp),\quad B_{12}\in\M_{r\times(n-r)}(\Cp),\quad B_{21}\in\M_{(n-r)\times r}(\Cp),\quad B_{22}\in\M_{n-r}(\Cp)\]
	and vectors $v_1\in(\Cp)^r$ and $v_2\in(\Cp)^{n-r}$ such that $B_{11}$ is invertible and
	\[B=\begin{pmatrix}
		B_{11} & B_{12} \\
		B_{21} & B_{22}
	\end{pmatrix}\text{ and }
	v_0=\begin{pmatrix}
		v_1 \\ v_2
	\end{pmatrix}.\]
	Let $v_3=-B_{11}^{-1}B_{12}v_2\in(\Cp)^r$ and $v=(v_3,v_2)\in(\Cp)^n$. Then we have
	\begin{equation}\label{eq:B}
		\begin{pmatrix}
			B_{11} & B_{12}
		\end{pmatrix}v=B_{11}v_3+B_{12}v_2=0.
	\end{equation}
	Since the rank of $B$ is $r$, its first $r$ rows generate all rows, hence expression \eqref{eq:B} implies $Bv=0$, and $Av=0$ because the rows of $A$ are multiples of those of $B$.
	
	\medskip
	\textit{Step 4: Conclusion.} We now have that
	\begin{align*}
		\|v-v_0\|_p & =\|(-B_{11}^{-1}B_{12}v_2-v_1,0)\|_p \\
		& =\|B_{11}^{-1}B_{12}v_2+v_1\|_p \\
		& =\|B_{11}^{-1}(B_{12}v_2+B_{11}v_1)\|_p.
	\end{align*}
	Since all entries of $B_{11}$ have $p$-adic absolute value at most $1$ and its determinant has $p$-adic absolute value $c$, the $p$-adic absolute value of the entries of $B_{11}^{-1}$ is at most $1/c$, so
	\begin{align*}
		\|v-v_0\|_p & \le \frac{1}{c}\|B_{12}v_2+B_{11}v_1\|_p \\
		& =\frac{1}{c}\max_{1\le i\le r}|B_{i\cdot}v_0|_p \\
		& =\frac{1}{c}\max_{1\le i\le r}\frac{|A_{i\cdot}v_0|_p}{\|A_{i\cdot}\|_p}.\qedhere
	\end{align*}
\end{proof}

Theorem \ref{thm:eigen} is a consequence of Lemma \ref{lemma:main} applied to the matrix $\lambda \mathrm{I}_n-A$. A vector such that $(\lambda \mathrm{I}_n-A)v=0$ is exactly an eigenvector of $A$ corresponding to $\lambda$.

\subsection{Proof of Theorem \ref{thm:integrable}}

We prove a more general result which implies Theorem \ref{thm:integrable}:

\begin{theorem}\label{thm:integrable2}
	\letnpos. \letpprime. Let $(M,\omega)$ be a $p$-adic analytic symplectic manifold of dimension $2n$. Let $F=(f_1,\ldots,f_n):(M,\omega)\to(\Qp)^n$ be a $p$-adic analytic integrable system. Let $m$ be a non-degenerate critical point of $F$. Let $a_1,\ldots,a_n\in\Qp$ such that $m$ is a non-degenerate critical point of $\sum_{i=1}^n a_if_i$ in the symplectic sense. Let $\Omega$ be the matrix of $\omega_m$ and $A=\Omega^{-1}\dd^2 \sum_{i=1}^n a_if_i$. Suppose that there exists an eigenvalue $\lambda$ of $A$ such that $\lambda\notin\Qp$ and $\lambda^2\in\Qp$. Then the following hold.
	
	\begin{enumerate-roman}
		\item Let $X_p$ be as described in \eqref{eq:X}. Then there exist a unique parameter $c\in X_p$, an open neighborhood $U$ of $m$, and local symplectic coordinates $(x_1,\xi_1,\ldots,x_n,\xi_n)$ on $U$ centered at $m$ such that, in these coordinates, one of the components of $F$ has the form $(x_1,\xi_1,\ldots,x_n,\xi_n)\mapsto x_1^2+c\xi_1^2$.
		\item Suppose that there exists $v_0\in(\Cp)^{2n}$ which is an $(A,\lambda,\Omega)$-almost eigenvector according to Definition \ref{def:almost} with $\ell=2n$. Let $x\mapsto\bar{x}$ be an automorphism of $\Cp$ that fixes $\Qp$ and sends $\lambda$ to $-\lambda$. Let $k=v_0\tr\Omega \bar{v}_0$.
		The parameter $c\in X_p$ in part \textup{(i)} is the element $c\in X_p$ for which the $p$-adic biquadratic polynomial
		\[p_{c,\lambda,k}(r,s)=(ck^2)r^4+(2c^2k^2)r^2s^2+(c^3k^2)s^4+4\lambda^4\]
		has a zero in $(\Qp)^2$.
	\end{enumerate-roman}
\end{theorem}

\begin{proof}
	We start with the proof of part (i), which we divide in four steps.

\medskip
\textit{Step 1: reduce to aligned local symplectic coordinates and rank $0$.} Let $r$ be the rank of $m$. By Lemma \ref{lemma:aligned}, we can change $F$ to a separated form $F'$ at $m$ such that \[\dd f_1'(m)=\ldots=\dd f_{n-r}'(m)=0\] and choose aligned local symplectic coordinates $(x_1,\xi_1,\ldots,x_n,\xi_n)$ for $F'$ around $m$; from this point on $\omega=\omega_0$ and \[\dd f_i'(m)=\dd \xi_i\] for all $i$ with $n-r<i\le n$. This reduces the problem to working in dimension $2(n-r)$; we assume in the following that $r=0$.

\medskip
\textit{Step 2: equality of eigenvectors between matrices.} Again without loss of generality, we assume that $a_1=1$ and $a_i=0$ for all $i>1$. Let \[A_i=\Omega_0^{-1}\dd^2 f_i'(m)\] in the coordinates $(x_1,\xi_1,\ldots,x_n,\xi_n)$. Since the functions form an integrable system,
\[\{f_i',f_j'\}=0\]
for all $1\le i,j\le n$, that is,
\[(\dd f_i')^T\Omega_0^{-1}\dd f_j'=0.\]
Differentiating this twice and evaluating at $m$,
\[\dd^2 f_i'(m)\Omega_0^{-1}\dd^2 f_j'(m)=\dd^2 f_j'(m)\Omega_0^{-1}\dd^2 f_i'(m)\]
This implies that \[A_iA_j=A_jA_i\] for all $1\le i,j\le n$.

Let $v\in(\Cp)^4$ be an eigenvector of $A_1$ with eigenvalue $\lambda$. Then
\[A_1A_iv=A_iA_1v=\lambda A_i v\]
for $1\le i\le n$. This implies that $A_iv$ is also an eigenvector of $A_1$ with value $\lambda$. But the critical point is non-degenerate, which means that the only eigenvector with value $\lambda$ is $v$. Hence, $A_iu=\mu v$ for some $\mu\in\Cp$, and $v$ is also an eigenvector of $A_i$. The same happens with $\lambda$ replaced by $-\lambda$; now the eigenvector is $\bar{v}$.

\medskip
\textit{Step 3: find a normal form for a matrix.} The next step is to follow part of the proof of \cite[Theorem 5.11]{CrePel-williamson2} but generalizing it to higher dimensions. We write $\lambda$ in the form $a\sqrt{-c}$ for $a\in\Qp$ and $c\in X_p$. There is always at least one $c$ for which this is possible, and there may be two. Now we apply \cite[Corollary 2.13]{CrePel-williamson2} to the matrix $\dd^2 f_1$, resulting in a symplectic matrix $S\in\M_{2n}(\Cp)$ such that
\[S^T\dd^2f_1S=\begin{pmatrix}
	a & 0 & 0 \\
	0 & ca & 0 \\
	0 & 0 & N
\end{pmatrix}\]
for some matrix $N\in\M_{2n-2}(\Cp)$. The matrix $S$ must have the form $\Psi_1D\Psi_2^{-1}$ where $D$ is a diagonal matrix, the first two columns of $\Psi_1$ are $v$ and $\bar{v}$, and
\[\Psi_2=\begin{pmatrix}
	\lambda & -\lambda & 0 \\
	a & a & 0 \\
	0 & 0 & \Psi'_2
\end{pmatrix}\]
where $\Psi'_2$ is some matrix that depends on $N$. If $d_1$ and $d_2$ are the first two entries of the diagonal of $D$ and $S_1$ is the matrix formed by the first two columns of $S$, we have that
\[S_1=\begin{pmatrix}
	v & \bar{v}
\end{pmatrix}
\begin{pmatrix}
	d_1 & 0 \\
	0 & d_2
\end{pmatrix}
\begin{pmatrix}
	\lambda & -\lambda \\
	a & a
\end{pmatrix}^{-1}.\]

Now recall \cite[Definition 3.4]{CrePel-williamson2} that if $F$ is a field and $c\in F$, then \[\DSq(F,c)=\{x^2+cy^2:x,y\in F\}.\] We will use this notation next.

By \cite[Proposition 3.6]{CrePel-williamson2}, $S$ has entries in $\Qp$ if and only if
\[\frac{2a\lambda}{v^T\Omega_0\bar{v}}\in\DSq(\Qp,-\lambda^2).\]
By \cite[Proposition 3.10]{CrePel-williamson2}, there is one and only one $c\in X_p$ for which this happens. If $S$ has entries in $\Qp$, it determines a linear symplectic coordinate change which makes $f_1$ of the form \[a(x_1^2+c\xi_1^2)+\ocal(x_1,\xi_1)^3+g_1(x_2,\xi_2,\ldots,x_n,\xi_n),\]
where the Hessian of $g_1$ is $N$.

\medskip
\textit{Step 4: find a normal form for the integrable system.} Since $S$ only depends on the eigenvectors of $A_1$, and the rest of $A_i$ have the same eigenvectors, they will have the same matrix $S$ and the same coordinate change will make $f_i$ of the form
\[a_i(x_1^2+c\xi_1^2)+\ocal(x_1,\xi_1)^3+g_i(x_2,\xi_2,\ldots,x_n,\xi_n)\]
in such a way that the Hessians of the functions $g_1,\ldots,g_n$ span a linear space of dimension $n-1$. (This happens because the matrix $N$ depends linearly on $n-1$ eigenvalues of $A_1$. Varying these eigenvalues will only change $N$ inside an $(n-1)$-dimensional linear space.)

In this situation, we can find a linear combination of the $f_i$ which equals
\[x_1^2+c\xi_1^2+\ocal(x_1,\xi_1,\ldots,x_n,\xi_n)^3,\]
and the proof of part (i) is complete.

\medskip
We now prove part (ii). We have that $c$ is the only element of $X_p$ such that $\lambda^2=-ca^2$ for some $a\in\Qp$ and
\begin{align*}
	\frac{2a\lambda}{v\tr\Omega \bar{v}}\in\DSq(\Qp,-\lambda^2) & =\DSq(\Qp,ca^2) \\
	& =\DSq(\Qp,c).
\end{align*}
We must show that this is exactly when $p_{c,\lambda,k}$ has a zero.

Applying Theorem \ref{thm:eigen} to the matrix $A$ and the subset $I$ in Definition \ref{def:almost}, we get that
\begin{align*}
	\|v-v_0\|_p & \le \frac{1}{\Ind_p(\lambda \mathrm{I}_\ell-A)_{I\cdot}}\max_{i \in I}\frac{|\lambda v_i-A_{i\cdot}v_0|_p}{\|\lambda \mathrm{e}_i-A_{i\cdot}\|_p} \\
	& =h(v_0)\|v_0\|_p.
\end{align*}
Now
\begin{align*}
	|v\tr\Omega \bar{v}-k|_p & =|v\tr\Omega \bar{v}-v_0\tr\Omega\bar{v}_0|_p \\
	& =|v\tr\Omega \bar{v}-v\tr\Omega \bar{v}_0+v\tr\Omega \bar{v}_0-v_0\tr\Omega\bar{v}_0|_p \\
	& =|v\tr\Omega(\bar{v}-\bar{v}_0)+(v-v_0)\tr\Omega\bar{v}_0|_p \\
	& =|(v-v_0)\tr\Omega(\bar{v}-\bar{v}_0)+v_0\tr\Omega(\bar{v}-\bar{v}_0)+(v-v_0)\tr\Omega\bar{v}_0|_p \\
	& \le \max\Big\{|(v-v_0)\tr\Omega(\bar{v}-\bar{v}_0)|_p,|v_0\tr\Omega(\bar{v}-\bar{v}_0)|_p,|(v-v_0)\tr\Omega\bar{v}_0|_p\Big\} \\
	& \le \max\Big\{k\|v-v_0\|_p^2,k\|v_0\|_p\|v-v_0\|_p\Big\} \\
	& \le \max\Big\{kh(v_0)^2\|v_0\|_p^2,kh(v_0)\|v_0\|_p^2\Big\} \\
	& =kh(v_0)\|v_0\|_p^2 \\
	& \le \frac{|v_0\tr\Omega \bar{v}_0|_p}{q}.
\end{align*}

This means that $v\tr\Omega \bar{v}$ and $k$ coincide in the order and one leading digit if $p\ne 2$, or in the order and three leading digits if $p=2$. Then their quotient is a square in $\Qp$ and
\[\frac{2a\lambda}{v\tr\Omega \bar{v}}\in\DSq(\Qp,c)\iff\frac{2a\lambda}{k}\in\DSq(\Qp,c).\]
This in turn is equivalent to the existence of $r,s\in\Qp$ such that
\[\frac{2a\lambda}{k}=r^2+cs^2\iff a=\frac{k(r^2+cs^2)}{2\lambda}.\]
Hence the $c\in X_p$ which we want is the one for which exist $r,s\in\Qp$ such that
\[\lambda^2=-ca^2=\frac{-ck^2(r^2+cs^2)}{4\lambda^2},\]
that is, $p_{c,\lambda,k}(r,s)=0$, as we wanted.
\end{proof}

\section{Low-dimensional cases: stability, independence between fields and proof of Theorems \ref{thm:indep2} and \ref{thm:indep3}}\label{sec:lowdim}

In this section we give a more general result than Theorem \ref{thm:integrable} in dimension $4$, Proposition \ref{prop:integrable}, which covers all possible cases with rank $0$. Then we use this result to prove that $p$-adic local normal forms in dimension $2$ and $4$ are stable, that is, making a small perturbation to a $p$-adic integrable system does not change the normal form (as long as the critical point is still a critical point with the same rank). In turn, this result is used to prove Theorem \ref{thm:indep2}.

\subsection{Computing normal forms in dimension $4$}

We know the complete Weierstrass-Williamson classification of $p$-adic critical points in dimension $4$. We can give in this case a result that, similarly to Theorem \ref{thm:integrable}, allows us to compute normal forms of critical points using only an approximation to the eigenvectors, but, unlike Theorem \ref{thm:integrable}, includes all possible types of rank $0$ critical points. We only need rank $0$ because, in dimension $4$, critical points can only have rank $0$ or $1$, and the case of rank $1$ is solved in Theorem \ref{thm:integrable1}. The cases in Proposition \ref{prop:integrable} correspond to the cases of Theorem \ref{thm:williamson}; some of them can be fully classified without the eigenvectors.

\begin{proposition}\label{prop:integrable}
	\letpprime. Let $(M,\omega)$ be a $p$-adic analytic symplectic manifold of dimension $4$. Let \[F=(f_1,f_2):(M,\omega)\to(\Qp)^2\] be a $p$-adic analytic integrable system. Let $m$ be a rank $0$ non-degenerate critical point of $F$ and let $g$ be a linear combination of $f_1$ and $f_2$ such that $m$ is a non-degenerate critical point of $g$ in the symplectic sense. Let $\Omega$ be the matrix of $\omega_m$ and $A=\Omega^{-1}\dd^2 g$. Let $\lambda$ be an eigenvalue of $A$. Then there exist an open neigborhood $U$ of $M$ and coordinates $(x,\xi,y,\eta)$ on $U$ centered at $M$ such that $\omega=\dd x\wedge\dd\xi+\dd y\wedge\dd\eta$ and one of the following cases holds.
	\begin{enumerate}
		\item[(1a)] Suppose that $\lambda\in\Qp$. Then one of the two components of $F$ at $m$ has the form $(x,\xi,y,\eta)\mapsto x^2+c\xi^2$ where $c=1$ if $p\equiv 1\mod 4$ and $c=-1$ otherwise.
		\item[(1b)] Suppose that $\lambda\notin \Qp$, $\lambda^2\in \Qp$ and there is $v_0\in(\Cp)^4$ which is an $(A,\lambda,\Omega)$-almost eigenvector according to Definition \ref{def:almost} with $\ell=4$. Let $c\in X_p$. Let $x\mapsto\bar{x}$ be an automorphism of $\Cp$ that fixes $\Qp$ and sends $\lambda$ to $-\lambda$. Let $c\in X_p$. One of the two components of $F$ at $m$ has the form $(x,\xi,y,\eta)\mapsto x^2+c\xi^2$ if and only if $\lambda^2=-ca^2$ for some $a\in \Qp$ and
		\[\frac{2a\lambda}{v_0\tr\Omega \bar{v}_0}\in\DSq(\Qp,c).\]
		\item[(2)] Suppose that $\lambda^2\notin\Qp$ and $\lambda\in\Qp[\lambda^2]$. Let $c\in Y_p$. $F$ has the form $(x,\xi,y,\eta)\mapsto(x\eta+cy\xi,x\xi+y\eta)$ if and only if $\lambda^2\in\Qp[\sqrt{c}]$.
		\item[(3)] Suppose that $\lambda^2\notin\Qp$, $\lambda\notin\Qp[\lambda^2]$ and there is $v_0\in(\Cp)^4$ which is an $(A,\lambda,\Omega)$-almost eigenvector according to Definition \ref{def:almost}. Let $x\mapsto\bar{x}$ be an automorphism of $\Cp$ that fixes $\Qp$ and sends $\lambda$ to $-\lambda$. $F$ has the form \eqref{eq:class3} with parameters $(c,t_1,t_2,a,b)$ if and only if $\lambda^2\in\Qp[\sqrt{c}]$, $\lambda^2/(t_1+t_2\sqrt{c})$ is a square in $\Qp[\sqrt{c}]$ and
		\[\frac{a\sqrt{c}(b+\sqrt{c})\sqrt{t_1+t_2\sqrt{c}}}{v_0\tr\Omega \bar{v}_0}\in\DSq(\Qp[\sqrt{c}],-t_1-t_2\sqrt{c}).\]
	\end{enumerate}
\end{proposition}

\begin{proof}
	Parts (1a) and (2) follow from \cite[Propositions 3.1 and 4.1]{CrePel-williamson2} respectively. Part (1b) is the particular case of Theorem \ref{thm:integrable} when the dimension is $4$.
	
	We finish with part (3). We apply \cite[Proposition 4.2]{CrePel-williamson2} to the matrix $\dd^2 g$, with the parameters $(c,t_1,t_2,a,b)$. We obtain that the normal form occurs if and only if
	\[\frac{a\sqrt{c}(b+\sqrt{c})\sqrt{t_1+t_2\sqrt{c}}}{v\tr\Omega \bar{v}}\in\DSq(\Qp[\sqrt{c}],-t_1-t_2\sqrt{c}).\]
	What is left to prove is that $v\tr\Omega \bar{v}$ is in $\DSq(\Qp[\sqrt{c}],-t_1-t_2\sqrt{c})$ if and only if $v_0\tr\Omega\bar{v}_0$ is there. Given $x,y\in\Qp[\sqrt{c}]$, whether $x\in\DSq(\Qp[\sqrt{c}],y)$ only depends on the class of $x$ and $y$ modulo squares in $\Qp[\sqrt{c}]$, which in turn, for a fixed $y$, only depends on the order and one leading digit of $x$, if $p\ne 2$, or the order and three leading digits of $x$, if $p=2$. Hence, what we want to prove is that $v\tr\Omega \bar{v}$ and $v_0\tr\Omega\bar{v}_0$ coincide in the order and that number of leading digits (here ``leading digits'' should be understood as having the form $r+s\sqrt{c}$ for $r,s\in\{0,\ldots,p-1\}$). From this point on the argument is the same as in part (ii) of Theorem \ref{thm:integrable}.
\end{proof}

\subsection{Stability of normal forms in dimension $2$ and $4$}

\begin{lemma}\label{lemma:polynomials}
	\letpprime. Let $f\in\Qp[x]$ of degree $\ell$ without multiple roots and let $\lambda_1,\ldots,\lambda_\ell\in\overline{\Qp}$ be its roots. Every polynomial in $\Qp[x]$ of degree $\ell$ which is close enough to $f$, in the sense of $p$-adic distance of the coefficients, has roots $\mu_1,\ldots,\mu_\ell\in\overline{\Qp}$ such that $\Qp[\lambda_i]=\Qp[\mu_i]$ for every $i\in\{1,\ldots,\ell\}$.
\end{lemma}

\begin{proof}
	Let $g$ be a polynomial very close to $f$ and let $\mu_1,\ldots,\mu_\ell\in\overline{\Qp}$ be its roots. These roots must be near $\lambda_1,\ldots,\lambda_\ell$. We multiply $x$ by a constant in order to make $|\lambda_i-\lambda_j|_p\ge 1$ and $|\mu_i-\mu_j|_p\ge 1$ for every $i,j\in\{1,\ldots,\ell\}$.
	
	Since the roots are all simple, $f'(\lambda_i)\ne 0$ and $g'(\mu_i)\ne 0$. If $f$ and $g$ are close enough, for all $i$ we will have $|\lambda_i-\mu_i|_p<1$. We fix $i\in\{1,\ldots,\ell\}$ and make a translation to $x$ such that $|\lambda_i|_p<1$ and $|\mu_i|_p<1$. After that, we multiply $f$ and $g$ by a constant in order to make $f,g\in\Zp[x]$. Now, if $f$ and $g$ are close enough,
	\[|f(\mu_i)|_p<|f'(\mu_i)|_p^2\text{ and }|g(\lambda_i)|_p<|g'(\lambda_i)|_p^2.\]
	Then we can apply Hensel's lemma to $f$ as a polynomial in $\Zp[\mu_i][x]$ at the point $\mu_i$, which is an integer point of $\Zp[\mu_i]$ (because its absolute value is less than $1$), and conclude that $f$ must have a root in $\Zp[\mu_i]$ at distance less than $1$ of $\mu_i$. This root must be $\lambda_i$, because any other root of $f$ is at distance at least $1$. This means that $\lambda_i\in\Qp[\mu_i]$. Analogously, $\mu_i\in\Qp[\lambda_i]$, and the proof is complete.
\end{proof}

\begin{theorem}\label{thm:stability}
	\letpprime. Let $n\in\{1,2\}$. Let $(M,\omega)$ be a $p$-adic analytic symplectic manifold of dimension $2n$. Let $F:(M,\omega)\to(\Qp)^n$ be a $p$-adic analytic integrable system. Let $m$ be a non-degenerate critical point of $F$. Then every small perturbation of $F$ which still has $m$ as a critical point of the same rank has the same normal form as $F$.
\end{theorem}

\begin{proof}
	Let $g:(M,\omega)\to\Qp$ be a linear combination of the functions of $F$ such that $m$ is a critical point of $g$. If the rank of $m$ is $0$, we can choose $g$ such that this critical point is non-degenerate (if $n=1$, $g=F$; if $n=2$, this exists by definition of non-degenerate critical point); if the rank is $1$, after changing to aligned symplectic coordinates $(x,\xi,y,\eta)$, $m$ is a non-degenerate critical point of $g$ when restricted to $y=\eta=0$ as a function of two variables $(x,\xi)$.
	
	Let $F'$ be a small perturbation of $F$ and $g'$ the corresponding perturbation of $g$. Let $\Omega$ be the matrix of $\omega$, $A=\Omega^{-1}\dd^2 g(m)$ and $A'=\Omega^{-1}\dd^2 g'(m)$. Lemma \ref{lemma:polynomials}, applied to the characteristic polynomials of $A$ and $A'$, implies that their eigenvalues generate the same extensions of $\Qp$. If the corank of $m$ is $1$, by Theorem \ref{thm:integrable1}, the normal form is determined by the extensions $\Qp[\lambda]$ for $\lambda$ eigenvalue of $A$; hence, $F$ and $F'$ have the same normal form. If the corank is $2$, the eigenvectors of $A$ are almost eigenvectors of $A'$. By Proposition \ref{prop:integrable}, the normal form is determined by the extensions $\Qp[\lambda]$ and the almost eigenvectors of $A$; hence, $F$ and $F'$ have again the same normal form.
\end{proof}

\subsection{Independence of normal forms between fields}

When the dimension is $2$, our result of independence between fields is a consequence of the results in the previous section.

\begin{lemma}\label{lemma:remainder}
	Let $p_1,\ldots,p_k$ be different prime numbers, $r_1,\ldots,r_k\in\N$ and $c_i\in\Q_{p_i}$ for $i\in\{1,\ldots,k\}$. There exists $c\in\Q$ such that
	\[|c-c_i|_{p_i}\le p_i^{-r_i}\]
	for all $i\in\{1,\ldots,k\}$.
\end{lemma}

\begin{proof}
	For $i\in\{1,\ldots,k\}$, let $s_i=\ord_{p_i}(c_i)$, $P=p_1^{s_1}\ldots p_k^{s_k}\in\Q$ and $a_i=c_i/P\in\Q_{p_i}$. We have that
	\[\ord_{p_i}(a_i)=\ord_{p_i}(c_i)-\ord_{p_i}(P)=s_i-s_i=0,\]
	hence $a_i\in\Z_{p_i}$ for all $i$. By the Chinese Remainder Theorem, there exists $a\in\Z$ such that $a\equiv a_i\mod p_i^{r_i-s_i}$ for all $i\in\{1,\ldots,k\}$ with $r_i>s_i$. We take $c=Pa$.
	If $r_i\le s_i$,
	\begin{align*}
		|c-c_i|_{p_i} & \le\max\{|c|_{p_i},|c_i|_{p_i}\} \\
		& =\max\{|P|_{p_i}|a|_{p_i},p_i^{-s_i}\} \\
		& =\max\{p_i^{-s_i}|a|_{p_i},p_i^{-s_i}\} \\
		& =p_i^{-s_i} \\
		& \le p_i^{-r_i}.
	\end{align*}
	The fourth line uses that $a\in\Z$. If $r_i>s_i$,
	\begin{align*}
		|c-c_i|_{p_i} & =|P|_{p_i}|a-a_i|_{p_i} \\
		& =p^{-s_i}|a-a_i|_{p_i} \\
		& \le p^{-s_i}p^{s_i-r_i} \\
		& =p^{-r_i},
	\end{align*}
	and we are done.
\end{proof}

\begin{theorem}\label{thm:indep1}
	Let $p_1,\ldots,p_k$ be different primes. Let $f_0:(\R^2,\dd x\wedge\dd\xi)\to\R$ be a non-degenerate integrable real normal form and let $f_i:((\Q_{p_i})^2,\dd x\wedge\dd\xi)\to\Q_{p_i}$ a non-degenerate integrable $p_i$-adic normal form for each $i\in\{1,\ldots,k\}$. Then there exists a polynomial $f$ with a critical point at the origin whose integrable real normal form is $f_0$ and whose integrable $p_i$-adic normal form is $f_i$ for each $i\in\{1,\ldots,k\}$.
\end{theorem}

\begin{proof}
	Each normal form $f_i$ can be written as $x^2+c_i\xi^2$ for some $c_i\in X_{p_i}$. For each $i\in\{1,\ldots,k\}$, by Theorem \ref{thm:stability}, there exists $r_i\in\N$ such that $x^2+c\xi^2$ has the same normal form for all $c\in\Q_{p_i}$ with $|c-c_i|_{p_i}\le p_i^{-r_i}$. By Lemma \ref{lemma:remainder}, there exists $c_0\in\Q$ such that $|c_0-c_i|_{p_i}\le p_i^{-r_i}$ for all $i$.
	
	On the other hand, $f_0$ can be the elliptic normal form $(x^2+\xi^2)/2$ or the hyperbolic normal form $x\xi$. If it is elliptic and $c_0>0$ or it is hyperbolic and $c_0<0$, we take $c=c_0$. Otherwise, we obtain $c$ from $c_0$ adding or subtracting a multiple of $p_1^{r_1}\ldots p_k^{r_k}$ which makes $c$ of the opposite sign of $c_0$. In any case, $x^2+c\xi^2$ has $f_0$ as an integrable real normal form (elliptic if $c>0$ or hyperbolic otherwise). Moreover,
	\[|c-c_i|_{p_i}\le\max\{|c-c_0|_{p_i},|c_0-c_i|_{p_i}\}\le p_i^{-r_i},\]
	which implies that the integrable $p_i$-adic normal form of $x^2+c\xi^2$ is $f_i$, as we wanted.
\end{proof}

In order to extend this result to order $4$, we first need to give a new expression to the normal forms in parts (1) and (2) of Theorem \ref{thm:williamson}. It turns out that the polynomials $g_k$ used for part (3), when generalized to values of $(c,t_1,t_2,a,b)$ not used in the normal forms, can also take all normal forms in parts (1) and (2).

\begin{lemma}\label{lemma:eigen}
	Let $g_1,g_2\in\Q(c,t_1,t_2,a,b)[x,\xi,y,\eta]$ be the polynomials in \eqref{eq:class3}. For any choice of $c,t_1,t_2,a,b\in\Q$ such that $c\ne 0$, $a\ne 0$, $t_1^2-ct_2^2\ne 0$ and $b^2-c\ne 0$, the following holds:
	\begin{enumerate}
		\item $(g_1,g_2):(\R^4,\dd x\wedge\dd\xi+\dd y\wedge\dd\eta)\to\R^2$ is an integrable system with a non-degenerate rank $0$ critical point in the origin;
		\item for any prime $p$, $(g_1,g_2):((\Qp)^4,\dd x\wedge\dd\xi+\dd y\wedge\dd\eta)\to(\Qp)^2$ is a $p$-adic integrable system with a non-degenerate rank $0$ critical point in the origin;
		\item the matrix $\Omega_0^{-1}\dd^2 g_1$ has eigenvalues $\lambda,-\lambda,\mu,-\mu$ and $\Omega_0^{-1}\dd^2 g_2$ has eigenvalues $\lambda\sqrt{c},-\lambda\sqrt{c},-\mu\sqrt{c},\mu\sqrt{c}$, where $\lambda=\sqrt{t_1+t_2\sqrt{c}}$ and $\mu=\sqrt{t_1-t_2\sqrt{c}}$;
		\item both matrices have eigenvectors
		\begin{gather*}
			(\lambda(b+\sqrt{c}),a\sqrt{c},a\lambda\sqrt{c}(b+\sqrt{c}),1), \\
			(-\lambda(b+\sqrt{c}),a\sqrt{c},-a\lambda\sqrt{c}(b+\sqrt{c}),1), \\
			(\mu(b-\sqrt{c}),-a\sqrt{c},-a\mu\sqrt{c}(b-\sqrt{c}),1), \\
			(-\mu(b-\sqrt{c}),-a\sqrt{c},a\mu\sqrt{c}(b-\sqrt{c}),1).
		\end{gather*}
	\end{enumerate}
\end{lemma}

\begin{proof}
	Parts (3) and (4) were proved in \cite[Proposition 4.2]{CrePel-williamson2} for any field with characteristic different from $2$.
	
	If $A_i=\Omega_0^{-1}\dd^2 g_i$, this matrix does not depend on the point $(x,\xi,y,\eta)$, and the fact that $(g_1,g_2)$ is an integrable system, both real and $p$-adic, is equivalent to saying that $A_1A_2=A_2A_1$ and that $A_1$ and $A_2$ are linearly independent. The former holds because both matrices have the same eigenvectors and the latter holds for any choice of $c,t_1,t_2,a,b\in\Q$ with the mentioned properties.
	
	It is left to prove that the origin is a non-degenerate rank $0$ critical point. It is a rank $0$ critical point because the polynomials are homogeneous of degree $2$ and it is non-degenerate because the eigenvalues of $A_1$ and $A_2$ are pairwise distinct.
\end{proof}

\begin{corollary}\label{cor:realtype}
	Let $(g_1,g_2):(\R^4,\dd x\wedge\dd\xi+\dd y\wedge\dd\eta)\to\R^2$ be the integrable system in Lemma \ref{lemma:eigen}(1). The Williamson type of the origin is
	\begin{itemize}
		\item elliptic-elliptic if $c>0$, $t_1+t_2\sqrt{c}<0$ and $t_1-t_2\sqrt{c}<0$;
		\item elliptic-hyperbolic if $c>0$ and $t_1+t_2\sqrt{c}$ and $t_1-t_2\sqrt{c}$ have opposite signs;
		\item hyperbolic-hyperbolic if $c>0$, $t_1+t_2\sqrt{c}>0$ and $t_1-t_2\sqrt{c}>0$;
		\item focus-focus if $c<0$.
	\end{itemize}
\end{corollary}

\begin{proof}
	The Williamson type of a critical point in the real case is entirely determined by the eigenvalues of the matrices $\Omega_0^{-1}\dd^2 g_i$ for $i=1,2$. If $c<0$, the eigenvalues are two complex conjugate pairs, hence the origin is a focus-focus point. If $c>0$, the pair of eigenvalues $\lambda,-\lambda$ is real if and only if $t_1+t_2\sqrt{c}>0$ and the pair $\mu,-\mu$ is real if and only if $t_1-t_2\sqrt{c}<0$. A pair of opposite eigenvalues gives an elliptic component if they are imaginary and a hyperbolic component if they are real.
\end{proof}

\begin{corollary}\label{cor:conversion}
	Let $p$ be a prime number. Let $(g_1,g_2):((\Qp)^4,\dd x\wedge\dd\xi+\dd y\wedge\dd\eta)\to(\Qp)^2$ be the integrable system in Lemma \ref{lemma:eigen}(2).
	\begin{enumerate}
		\item Let $c_1,c_2\in X_p$. If \[(c,t_1,t_2,a,b)=\left(1,-\frac{c_1+c_2}{2},\frac{c_2-c_1}{2},\frac{1}{2},0\right),\] the normal form of $(g_1,g_2)$ is that of Theorem \ref{thm:williamson}(1) for those $c_1,c_2$.
		\item Let $c\in Y_p$. If \[(c,t_1,t_2,a,b)=(c,1,0,1,0),\] the normal form of $(g_1,g_2)$ is that of Theorem \ref{thm:williamson}(2) for that $c$.
	\end{enumerate}
\end{corollary}

\begin{proof}
	For the choice of parameters in part (1), the eigenvalues of $A_1$ are $\sqrt{-c_1}$, $-\sqrt{-c_1}$, $\sqrt{-c_2}$, $-\sqrt{-c_2}$ and those of $A_2$ are $\sqrt{-c_1}$, $-\sqrt{-c_1}$, $-\sqrt{-c_2}$, $\sqrt{-c_2}$. We can recombine the matrices in order to get a matrix with eigenvalues $\sqrt{-c_1}$, $-\sqrt{-c_1}$, $0,0$ and another one with $0,0$, $\sqrt{-c_2}$, $-\sqrt{-c_2}$. We use the criterion from Theorem \ref{thm:integrable} with $c=c_i$, $\lambda=\sqrt{-c_i}$ and
	\[k=\left(\frac{1}{2},\lambda,\frac{\lambda}{2},1\right)\tr\Omega_0\left(\frac{1}{2},-\lambda,-\frac{\lambda}{2},1\right)=2\lambda=2\sqrt{-c_i},\]
	which gives
	\[ck^2(r^2+cs^2)^2+4\lambda^4=-4c_i^2(r^2+c_is^2)^2+4c_i^2.\]
	This is $0$ for $r=1$ and $s=0$.
	
	Respect to part (2), by Proposition \ref{prop:integrable}(2), it is enough to check the eigenvalues of $A_1+A_2$, which are $1+\sqrt{c}$, $-1-\sqrt{c}$, $1-\sqrt{c}$ and $-1+\sqrt{c}$. If $\lambda=1+\sqrt{c}$, $\lambda^2=c+1+2\sqrt{c}\notin\Qp$ and $\lambda\in\Qp[\sqrt{c}]=\Qp[\lambda^2]$, hence the normal form is what we want.
\end{proof}

\subsection{Proof of Theorem \ref{thm:indep2}}

	If $r=1$, we take $f_2=\eta$. Then finding $f_1$ reduces to the case of dimension $2$ in the coordinates $(x,\xi)$ and applying Theorem \ref{thm:indep1} gives the result we want.
	
	Now assume that $r=0$. If a normal form $F_i$ is one of those in parts (1) or (2), by Corollary \ref{cor:conversion} we can write it in the form $(g_{i1},g_{i2})$, where $g_{i1}$ and $g_{i2}$ are a particular case of $g_1$ and $g_2$ in \eqref{eq:class3} for some parameters $(c_i,t_{i1},t_{i2},a_i,b_i)$. If $F_i$ is one of those in part (3), it is already in this form.
	
	For each $i\in\{1,\ldots,k\}$, by Lemma \ref{lemma:eigen}(1) and (2), there exists $r_i\in\N$ for each $i$ such that, for all $(c,t_1,t_2,a,b)\in(\Q_{p_i})^5$ with
	\begin{equation}\label{eq:radius}
		\|(c,t_1,t_2,a,b)-(c_i,t_{i1},t_{i2},a_i,b_i)\|_{p_i}\le p_i^{-r_i},
	\end{equation}
	the polynomials $(g_1,g_2)$ obtained with these parameters is both a real integrable system and a $p_i$-adic integrable system. By Theorem \ref{thm:stability}, we can increase $r_i$ so that this $p_i$-adic integrable system has $F_i$ as a normal form.
	
	By Lemma \ref{lemma:remainder}, there exists $(c_0,t_{01},t_{02},a,b)\in\Q^5$ which satisfies the inequalities \eqref{eq:radius} simultaneously for all $i$. We define $c,t_1,t_2\in\Q$ adding or subtracting a multiple of $p_1^{r_1}\ldots p_k^{r_k}$ to $c_0,t_{01},t_{02}$ which makes them satisfy the inequalities in the case of Corollary \ref{cor:realtype} corresponding to the normal form $F_0$; this means that the real integrable system $(g_1,g_2)$ obtained with these parameters will have $F_0$ as a normal form. Moreover, changing $c,t_1,t_2$ in this way does not affect \eqref{eq:radius}; hence, the integrable $p_i$-adic normal form will be $F_i$ for each $i$, as we wanted.
	
\subsection{Proof of Theorem \ref{thm:indep3}}

For each $j\in\{1,\ldots,s\}$, we apply Theorem \ref{thm:indep2} to the components $2j-1$ and $2j$ of each integrable system. This gives us two polynomials $f_{2j-1}$ and $f_{2j}$ in $\Q[x_{2j-1},\xi_{2j-1},x_{2j},\xi_{2j}]$. All this polynomials together form the integrable system we want.

\section{Final remarks}

\begin{remark}
	Throughout this paper we work with $p$-adic analytic manifolds in the sense of Serre \cite{Serre} and Schneider \cite{Schneider}, instead of other types of $p$-adic spaces which are important in algebraic geometry. Our viewpoint seems optimal to treat problems arising in classical and quantum mechanics (see for instance our works on the $p$-adic coupled angular momentum \cite{CrePel-angular} and the $p$-adic Jaynes-Cummings model \cite{CrePel-JC}) as well as in symplectic topology (see \cite{CrePel-nonsqueezing}). Note that we follow Schneider's viewpoint \cite[Section 8]{Schneider} and we do not require that symplectic manifolds are paracompact, as it is often required in the real case.
\end{remark}

\begin{remark}
	Theorem \ref{thm:integrable1} can use any value of $c$ which satisfies the condition. There are infinitely many possible values of $c$, because multiplying $c$ by $t^4$ for any $t\in\Qp$ will give a new solution. However, as said in Theorem \ref{thm:williamson}, the Weierstrass-Williamson classification selects a \emph{finite} set of forms as $p$-adic normal forms of critical points; only one value of $c$, among the infinitely many which are equivalent by symplectic transformation, will be contained in this finite set.
\end{remark}

\begin{remark}
	Same as in Theorem \ref{thm:integrable1}, infinitely many values of $c$ satisfy the condition in Theorem \ref{thm:integrable}, but only one is the normal form used for the classification in Theorem \ref{thm:williamson}. See Figure \ref{fig:fibers1} for a representation of the fiber of one of the normal forms. Theorem \ref{thm:integrable} allows us to compute the normal form of a concrete type of critical points (in any dimension and any rank) using \emph{an approximation} of the eigenvectors instead, hence giving us a potentially powerful tool to deal with systems from physics such as the Jaynes-Cummings model and the coupled angular momentum system, where exact information is either not available or difficult to find.
\end{remark}

\begin{remark}
	Integrable systems display features which can be encoded by algebraic, analytical of combinatorial invariants, and in some cases these invariants even ``classify'' the system (as it happens for example in the toric and semitoric cases \cite{Atiyah,Delzant,GuiSte,PPT,PelVuN-semitoric,PelVuN-construct}).
\end{remark}

\begin{remark}
	The classification of normal forms in dimensions $2$ and $4$ was solved in \cite{CrePel-williamson,CrePel-williamson2}; Theorem \ref{thm:williamson} is part of it. However, in those papers we did not address the problem of finding the normal form of a \textit{given} critical point of a $p$-adic analytic integrable system; we just gave a list of normal forms and proved that every critical point was locally symplectomorphic to one and only one of these. Theorems \ref{thm:integrable1}, \ref{thm:integrable} and Proposition \ref{prop:integrable} give techniques to compute the normal forms in practice.
\end{remark}

\begin{remark}
	The polynomial $p_{\alpha,\beta,\gamma,c}(r,s)$ can be written in a simplified way as
	\[p_{\alpha,\beta,\gamma,c}(r,s)=\alpha^2c(r^2+cs^2)^2-\alpha\gamma+\beta^2.\]
	This is the form in which we use it, both in the proof of Theorem \ref{thm:integrable3} and in the examples. The value of $c$ is restricted to $X_p$ in Theorem \ref{thm:integrable3} while not restricted in Theorem \ref{thm:integrable1}; the proof is the same in both cases, with the only difference that $c$ is unique in $X_p$ and not in $\Qp$.
\end{remark}

\appendix
\section{Classification of critical points in the real and $p$-adic cases}\label{sec:classification}

The most intriguing aspects of integrable systems are encoded in their critical points. Classifying the local behavior of the system at the critical points allows us to understand the evolution of the system near them; in the real case this behavior is described by the Weierstrass-Williamson classification \cite{Weierstrass,Williamson}. For example, rank $0$ critical points of a $4$-dimensional integrable system $G:(N,\sigma)\to\R^2$ can only be of four types: elliptic-elliptic, elliptic-hyperbolic, hyperbolic-hyperbolic and focus-focus. The local normal forms of $G$ at these points are respectively
\begin{equation}\label{eq:normal-forms0}
	\left(\frac{x^2+\xi^2}{2},\frac{y^2+\eta^2}{2}\right),\left(\frac{x^2+\xi^2}{2},y\eta\right),(x\xi,y\eta)\,,(x\eta-y\xi,x\xi+y\eta).
\end{equation}
The local normal forms of $G$ at rank $1$ critical points can be elliptic-regular and hyperbolic-regular, which are respectively given by
\begin{equation}\label{eq:normal-forms1}
	\left(\frac{x^2+\xi^2}{2},\eta\right),(x\xi,\eta).
\end{equation}

Now let us move, as proposed by Pelayo-Voevodsky-Warren \cite[Section 7]{PVW}, to the case where the manifold, the integrable system and the symplectic form have coefficients in the non-Archimedean field of $p$-adic numbers $\Qp$, where $p$ is a prime number. Like in the real case, all $4$-dimensional $p$-adic analytic symplectic manifolds are locally equivalent \cite[Theorem B]{CrePel-Darboux} to the standard $p$-adic space $((\Qp)^4,\dd x\wedge\dd\xi+\dd y\wedge\dd\eta),$ where $(x,\xi,y,\eta)$ are the standard coordinates on $(\Qp)^4$ (in fact the analog result is true in any dimension). In our previous work \cite{CrePel-williamson,CrePel-williamson2} we gave an explicit classification of the possible local models of a $p$-adic symplectic integrable system $F:(M,\omega)\to(\Qp)^2$ in standard local coordinates, where $\omega=\dd x\wedge\dd\xi+\dd y\wedge\dd\eta$, near any singularity. In great contrast to \eqref{eq:normal-forms0} and \eqref{eq:normal-forms1}, this classification \textit{includes hundreds of inequivalent local normal forms} \cite[Theorem A]{CrePel-williamson} in the case of rank $0$ points and up to $11$ forms in the case of rank $1$ critical points. For example, the local normal forms for a rank $1$ critical point given by $(x^2+p\xi^2,\eta)$ and $(x^2+p^2\xi^2,\eta)$ do not appear in the real case. We also gave partial results on the number of local normal forms for systems $F=(f_1,\ldots,f_n)$ in arbitrary dimension using the Hardy-Ramanujan formula from number theory \cite{HarRam}.

\section{Equivalence between notions of non-degeneracy for critical points}\label{sec:degeneracy}

Definition \ref{def:nondeg} is not the usual notion of non-degenerate critical point, but both notions are related, as the following result shows.

\begin{proposition}[{\cite[Proposition 4.2]{CrePel-williamson}}]\label{prop:nondeg}
	Let $n$ be a positive integer. \letpprime. Let $M$ be a $p$-adic analytic $2n$-dimensional manifold, $f:M\to \Qp$ a $p$-adic analytic function, and $m$ a critical point of $f$. Then the following are equivalent:
	\begin{enumerate}
		\item $m$ is a non-degenerate critical point of $f$ in the usual sense;
		\item There exists a linear symplectic form $\omega$ such that $m$ is a non-degenerate critical point of $f:(M,\omega)\to \Qp$ in the symplectic sense (Definition \ref{def:nondeg}).
		\item There exist infinitely many linear symplectic forms $\omega$ such that $m$ is a non-degenerate critical point of $f:(M,\omega)\to \Qp$ in the symplectic sense.
	\end{enumerate}
\end{proposition}

\section{Existence of aligned symplectic coordinates}\label{sec:proof-aligned}

\begin{lemma}[Existence of aligned symplectic coordinates]\label{lemma:aligned}
	\letnpos. \letpprime. Let $(M,\omega)$ be a $p$-adic analytic symplectic manifold of dimension $2n$. Let $F=(f_1,\ldots,f_n):(M,\omega)\to(\Qp)^n$ be a $p$-adic analytic integrable system. Let $m$ be a critical point of $F$. Then there exist an open set $U$ which contains $m$, a $p$-adic invertible matrix $B\in\M_n(\Qp)$ such that $F'=B\circ F$ is in separated form at $m$, and aligned local symplectic coordinates for $F'$ on $U$ centered at $m$. We call $F'$ \emph{a separated form of $F$ at $m$.}
\end{lemma}

\begin{proof}
	We start by showing that $F'$ exists. Let $r$ be the rank of $m$. Let $V$ be the vector subspace of $\Omega_0(M)$ generated by $(f_1,\ldots,f_n)$ and consider the linear map $\phi:V\to(\mathrm{T}_mM)^*$ given by $\phi(f)=\dd f(m)$. By definition of rank, the image of $\phi$ has dimension $r$, hence its kernel has dimension $n-r$. Let $(f_1',\ldots,f_{n-r}')$ be a basis of the kernel of $\phi$ and complete it to a basis $(f_1',\ldots,f_n')$ of $V$. This basis is the separated form which we want.

	Now we prove that there are aligned local symplectic coordinates for $F'$. By \cite[Theorem B]{CrePel-Darboux}, there exists an open set $U$ which contains $m$ and local symplectic coordinates on $U$. Now we make $r$ linear combinations of these coordinates, $\xi_{n-r+1},\ldots,\xi_n$, such that $\dd f_i'(m)=\dd\xi_i(m)$ for all $i>n-r$. Since the functions form an integrable system, \[\{\xi_i,\xi_j\}(m)=\{f_i',f_j'\}(m)=0\] for all $i,j>n-r$. This implies that $\xi_{n-r+1},\ldots,\xi_n$ are partial linear symplectic coordinates on the tangent space at $m$, and as such they can be completed to linear symplectic coordinates $(x_1,\xi_1,\ldots,x_n,\xi_n)$ on the tangent space. Since the original coordinates were already linear symplectic coordinates on this space, the linear change from the old to the new coordinates is symplectic. This in turn implies that the new coordinates are symplectic on all of $U$ like the old ones, hence these new coordinates are aligned local symplectic coordinates, and we are done.
\end{proof}

\bigskip
\noindent \textit{Funding.} The first author is funded by grant PID2022-137283NB-C21 of MCIN/AEI/ 10.13039/501100011033 / FEDER, UE. The second author was funded by a FBBVA (Bank Bilbao Vizcaya Argentaria Foundation) Grant for Scientific Research Projects with title \textit{From Integrability to Randomness in Symplectic and Quantum Geometry} (2022-2025).

\medskip
\noindent \textit{Acknowledgements.} The second author thanks the Dean of the School of Mathematical Sciences Antonio Br\'u and the Chair of the Department of Algebra, Geometry and Topology at the Complutense University of Madrid, Rutwig Campoamor, during the tenure of the FBBVA project (2022-2025), for their support and excellent resources he was provided with to carry out the project. He also thanks the University of Cantabria and Francisco Santos for the hospitality during July of 2025 when part of this paper was written.

\end{document}